\documentclass[10pt,a4paper]{amsart}

\usepackage{a4wide}

\usepackage[english]{babel}

\usepackage[pdftex]{graphicx}\pdfcompresslevel=9

\usepackage[colorlinks=true]{hyperref}

\usepackage[all]{xy}
\usepackage{amsmath,amsthm,amssymb,enumerate}
\usepackage{latexsym}
\usepackage{amscd}

\newtheorem{theorem}[subsection]{Theorem}
\newtheorem{thm}[subsection]{Theorem}
\newtheorem{defn}[subsection]{Definition}
\newtheorem{prop}[subsection]{Proposition}

\newtheorem{remark}[subsection]{Remark}

\theoremstyle{definition}
\newtheorem{example}[subsection]{Example}

\newcommand{\R}{\mathbb R}
\newcommand{\Q}{\mathbb Q}
\newcommand{\Z}{\mathbb Z}
\newcommand{\N}{\mathbb N}
\newcommand{\C}{\mathbb C}
\newcommand{\CP}{{\mathbb P}}

\DeclareMathOperator{\Diff}{Diff}
\DeclareMathOperator{\Ham}{Ham}
\DeclareMathOperator{\Hess}{Hess}
\DeclareMathOperator{\Det}{Det}
\DeclareMathOperator{\End}{End}

\DeclareMathOperator{\diag}{diag}
\DeclareMathOperator{\Id}{Id}

\begin{document}

\title[Toric 	K\"ahler Metrics]
{Toric 	K\"ahler Metrics: Cohomogeneity One Examples of Constant 
Scalar Curvature in Action-Angle Coordinates}
\author{Miguel Abreu}
\begin{thanks}
{Partially supported by the Funda\c{c}\~{a}o para a Ci\^{e}ncia e a Tecnologia
(FCT/Portugal).}
\end{thanks}
\address{Centro de An\'{a}lise Matem\'{a}tica, Geometria e Sistemas
Din\^{a}micos, Departamento de Matem\'atica, Instituto Superior T\'ecnico, Av.
Rovisco Pais, 1049-001 Lisboa, Portugal}.
\email{mabreu@math.ist.utl.pt}

\begin{abstract} 
In these notes, after an introduction to toric K\"ahler geometry, 
we present Calabi's family of $U(n)$-invariant extremal K\"ahler metrics
in symplectic action-angle coordinates and show that it actually contains,
as particular cases, many interesting cohomogeneity one examples of
constant scalar curvature.
\end{abstract}

\maketitle

\section{Introduction} \label{s:intro}

In 1982 Calabi~\cite{C2} constructed, in local complex coordinates, a general
$4$-parameter family of $U(n)$-invariant extremal K\"ahler metrics, which he
used to put an extremal K\"ahler metric on
\[
H^n_m := \CP ({\mathcal O}(-m)\oplus\C) \to \CP^{n-1}\,,
\]
for all $n,m\in\N$ and any possible K\"ahler cohomology class. In particular,
when $n=2$, on all Hirzebruch surfaces.

The main goal of these notes is to present Calabi's general family in local 
symplectic action-angle coordinates, using the set-up of~\cite{Abr1, Abr2} 
for toric K\"ahler geometry, and show that it actually
contains other interesting cohomogeneity one K\"ahler metrics as particular 
cases (see also~\cite{Ra}).
These include:
\begin{itemize}
\item[-] the Fubini-Study, flat and Bergman K\"ahler-Einstein metrics of
constant holomorphic sectional curvature (positive, zero and negative, resp.).
\item[-] the complete Ricci flat K\"ahler metric on the total space of
\[
{\mathcal O}(-n) \to \CP^{n-1}\,,
\]
for all $n\in\N$  and any possible K\"ahler cohomology class, constructed by
Calabi~\cite{C1} in 1979.
\item[-] the complete scalar flat K\"ahler metric on the total space of
\[
{\mathcal O}(-m) \to \CP^{n-1}\,,
\]
for all $m,n\in\N$ and any possible K\"ahler cohomology class, constructed for
$n=2$ by LeBrun~\cite{LeB} in 1988 and for $n>2$ by Pedersen-Poon~\cite{PePo} 
in 1991 (see also Simanca~\cite{Si}).
\item[-] the complete K\"ahler-Einstein metric with negative scalar 
curvature on the total space of the open disc bundle 
\[
{\mathcal D}(-m) \subset {\mathcal O}(-m)\longrightarrow \CP^{n-1}\,,
\]
for all $n<m\in\N$ and any possible K\"ahler cohomology class, constructed by 
Pedersen-Poon~\cite{PePo}.
\item[-] the complete constant negative scalar curvature K\"ahler metric on 
the total space of the open disc bundle 
\[
{\mathcal D}(-m) \subset {\mathcal O}(-m)\longrightarrow \CP^{n-1}\,,
\]
for all $n,m\in\N$ and any possible K\"ahler cohomology class, also 
constructed by Pedersen-Poon~\cite{PePo}.
\end{itemize}

Calabi's general family contains many other interesting cohomogeneity one special
K\"ahler metrics. Besides the Bochner-K\"ahler orbifold examples presented in~\cite{Abr3}, 
it contains for example a family of singular K\"ahler-Einstein metrics on certain 
$H^n_m$ that are directly related to the Sasaki-Einstein metrics constructed by Gauntlett-Martelli-Sparks-Waldram~\cite{GW1,GW2} in 2004 - see~\cite{Abr4}.

These notes are organized as follows. In section~\ref{s:troicsympmflds} we
give a basic introduction to symplectic geometry and discuss 
some fundamental features of toric symplectic manifolds.
Section~\ref{s:toricKmetrics} is devoted to toric K\"ahler metrics. After some
relevant linear algebra prelimaries, we explain how these can be parametrized
in action-angle coordinates via symplectic potentials for the associated
toric compatible complex structures, and discuss some important properties of
these symplectic potentials. In section~\ref{s:2d-metrics} we write down the 
symplectic potentials that give rise to toric constant (scalar) curvature metrics 
in real dimension $2$ and identify the underlying toric symplectic surfaces. 
This is a warm-up for section~\ref{s:calabi}, where we discuss in detail 
the above higher dimensional examples that arise in Calabi's general family of 
local $U(n)$-invariant extremal K\"ahler metrics.

\subsection{Acknowledgements} I thank the organizers of the 
XI International Conference on Geometry, Integrability and Quantization, 
Varna, Bulgaria, June 5--10, 2009, where this material was presented as 
part of a mini-course.

The work presented in Section~\ref{s:calabi} was carried out in January-June 
of 2001, while I was at the Fields Institute for Research in Mathematical Sciences,
Toronto, Canada. It was presented in seminar talks given at the Workshop on 
Hamiltonian Group Actions and Quantization, Fields Institute, June 4--13, 2001, 
and at Differentialgeometrie im Gro{\ss}en, Mathematisches Forschungsinsitut 
Oberwolfach, Germany, June 10--16, 2001. I thank the support and hospitality of
the Fields Institute and the organizers of those meetings.

\section{Toric Symplectic Manifolds} \label{s:troicsympmflds}

In this section we give a basic introduction to symplectic geometry and discuss 
some fundamental features of toric symplectic manifolds.

\subsection{Symplectic Manifolds}

\begin{defn} \label{def:sympmfld}
A \emph{symplectic manifold} is a pair $(B,\omega)$ where $B$ is a smooth
manifold and $\omega$ is a closed and non-degenerate $2$-form, i.e.
\begin{itemize}
\item[(i)] $\omega\in\Omega^2(B)$ is such that $d\omega = 0$ and
\item[(ii)] for any $p\in B$ and $0\ne X\in T_p B$, there exists $Y\in
  T_p B$ such that $\omega_p (X,Y) \ne 0$.
\end{itemize}
\end{defn}
The non-degeneracy condition (ii) implies that a symplectic manifold is 
always even dimensional. If $B$ has dimension $2n$, the non-degeneracy 
condition (ii) is equivalent to requiring that
\[
\omega^n \equiv \omega\wedge\cdots\wedge\omega \in \Omega^{2n}(B)\ \text{is a
  volume form.}
\]
Hence, a symplectic manifold $(B,\omega)$ is always oriented.

\begin{example}\label{ex:R2n}
The most basic example is $\R^{2n}$ with linear coordinates 
\[
(u_1, \ldots, u_n, v_1, \ldots, v_n)
\] 
and symplectic form 
\[
\omega_{\rm st} = d u \wedge d v := \sum_{j=1}^n d u_j \wedge d v_j\,.
\]
\end{example}

\begin{example} \label{ex:2d}
Any $2$-dimensional surface equipped with an area form is a symplectic 
manifold. For example, the sphere $S^2$ or any other compact
orientable surface $\Sigma_g$ of genus $g$.
\end{example}

\begin{example} \label{ex:prod}
If $(B_1,\omega_1)$ and $(B_2,\omega_2)$ are symplectic manifolds, then
\[
(B=B_1\times B_2, \omega = \omega_1\times\omega_2)
\] 
is also a symplectic manifold. Here, $\omega_1\times\omega_2$ means the sum of 
the pullbacks of the symplectic forms $\omega_1$ and $\omega_2$ from the factors 
$B_1$ and $B_2$. 
\end{example}

\begin{example} \label{ex:kahler}
The imaginary part of the hermitean metric on any K\"ahler manifold is
a symplectic form. Hence, any K\"ahler manifold is a symplectic
manifold. In particular, the complex projective space $\CP^n$ equipped
with its Fubini-Study form $\omega_{FS}$ is a symplectic manifold.
\end{example}

When $(B,\omega)$ is a compact symplectic manifold we have that
\[
\omega^n = \text{volume form} \Rightarrow 0\ne [\omega^n]\in H^{2n}(B,\R) 
\Rightarrow 0\ne [\omega]\in H^2(B,\R)\,.
\]
In particular, the spheres ${\mathbb S}^{2n}$ have no symplectic form when
$n>1$, since 
\[
H^2({\mathbb S}^{2n}, \R) = 0 \quad\text{when}\quad n>1\,.
\]

\subsection{Symplectomorphisms and Darboux's Theorem}

\begin{defn} \label{def:sympdiff}
Let $(B,\omega)$ be a symplectic manifold. A \emph{symplectomorphism} of
$B$ is a diffeomorphism $\varphi:B\to B$ such that
$\varphi^\ast(\omega)=\omega$. These form the \emph{symplectomorphism group},
a subgroup of $\Diff(B)$ that will be denoted by $\Diff(B,\omega)$.
\end{defn}

\begin{example} \label{ex:hamdiff}
Consider a symplectic manifold $(B,\omega)$ and let $h:B\to\R$ be a
smooth function on $B$. The non-degeneracy of $\omega$ implies that there
exists a unique vector field $X_h \in{\mathcal X}(B)$ such that $X_h
\lrcorner\,\omega = d h$. This vector field $X_h$ is called the 
\emph{Hamiltonian vector field} of the function $h$ 
and has the following fundamental property:
\[
\text{the flow $\varphi_t\equiv (X_h)_t : B\to B$ consists of
  symplectomorphisms of $B$.}
\]
This can be proved using Cartan's formula to compute
\[
{\mathcal L}_{X_h}\omega = X_h \lrcorner\,d\omega + d (X_h \lrcorner\,\omega)
= X_h \lrcorner\, 0 + d(d h) = 0\,.
\]
Hence, on a symplectic manifold $(B,\omega)$ any smooth function $h\in
C^\infty(B)$ gives rise, through the flow of the corresponding
Hamiltonian vector field $X_h\in{\mathcal X}(B)$, to a $1$-parameter group of
symplectomorphisms.
\end{example}

One can use the symplectomorphisms constructed in the previous example
to prove that:
\begin{itemize}
\item[(i)] the symplectomorphism group $\Diff(B,\omega)$ is always
  infinite-dimensional;
\item[(ii)] the action of $\Diff(B,\omega)$ on the manifold $B$ is always
  $k$-transitive, for any $k\in\N$;
\item[(iii)] in particular, any point of a symplectic manifold
  $(B,\omega)$ looks locally like any other point of $(B,\omega)$.
\end{itemize}
This last statement is made more precise in the following
\begin{thm} \emph{[Darboux]} Let $(B, \omega)$ be a symplectic manifold
  of dimension $2n$. Then, any point $p\in B$ has a neighborhood
  $U\subset B$ symplectomorphic to a neighborhood $V$ of the origin in
  $(\R^{2n}, \omega_{\rm st})$, i.e. there exists a diffeomorphism
\[
\text{$\phi:U\subset B \to
  V\subset\R^{2n}$ such that $\phi(p) = 0$ and $\phi^\ast(\omega_{\rm st}) =
  \omega$.} 
\]
\end{thm}
In other words,
\[
\text{there are no local invariants in symplectic geometry,}
\]
which is in sharp contrast with what happens, for example, in
Riemannian geometry.

\subsection{Symplectic and Hamiltonian Vector Fields}

The Lie algebra of the symplectomorphism group $\Diff(B,\omega)$, viewed
as an infinite-dimensional Lie group, is naturally identified with the
vector space ${\mathcal X}(B,\omega)$ of \emph{symplectic vector fields}, i.e. vector
fields $X\in{\mathcal X}(B)$ such that ${\mathcal L}_X \omega = 0$, with Lie bracket 
$\left[\cdot,\cdot\right]$ given by the usual Lie bracket of vector
fields. As before, we can use Cartan's formula to obtain
\[
{\mathcal L}_X \omega = X \lrcorner\, d\omega + d(X\lrcorner\,\omega) =
X \lrcorner\,0 + d(X\lrcorner\,\omega) = d(X\lrcorner\,\omega)\,.
\]
Hence, the vector space of symplectic vector fields is given by
\[
{\mathcal X}(B,\omega) = \left\{X\in{\mathcal X}(B)\,:\ \text{the $1$-form
    $X\lrcorner\,\omega$ is closed} \right\}\,,
\]
while its subspace of \emph{Hamiltonian vector fields} is given by
\[
{\mathcal X}_H (B,\omega) = \left\{X\in{\mathcal X}(B)\,:\ \text{the $1$-form
    $X\lrcorner\,\omega$ is exact} \right\}\,.
\]
In fact, as the following theorem shows, ${\mathcal X}_H (B,\omega)$ is a Lie
subalgebra of ${\mathcal X}(B,\omega)$.
\begin{thm} \label{thm:hamvf}
If $X,Y\in{\mathcal X}(B,\omega)$ are symplectic vector fields, then
$\left[X,Y\right]$ is the Hamiltonian vector field of the function
$\omega(Y,X) : B\to\R$, i.e.
\[
\left[X,Y\right] = X_{\omega(Y,X)} \in {\mathcal X}_H(B,\omega)\,.
\]
\end{thm}
\begin{proof}
It suffices to compute $\left[X,Y\right] \lrcorner\, \omega$, using
standard formulas from differential geometry and the defining
properties of $X$, $Y$ and $\omega$:
\[
\begin{array}{rcl}
[X,Y] \lrcorner\, \omega & = & {\mathcal L}_X (Y\lrcorner\,\omega) - Y\lrcorner\,
({\mathcal L}_X \omega) \\
& = & d(X\lrcorner\,(Y\lrcorner\,\omega)) + X\lrcorner\,(d(Y\lrcorner\,
\omega)) - Y \lrcorner\, (d(X\lrcorner\,\omega)) - Y \lrcorner\,
(X\lrcorner\, d\omega) \\
& = & d(\omega(Y,X))\,.
\end{array}
\]
\end{proof}
\begin{remark} \label{rmk:hamgp}
${\mathcal X}_H(B,\omega)$ is the Lie algebra of a fundamental subgroup of the
symplectomorphism group: the subgroup $\Ham(B,\omega)\subset
\Diff(B,\omega)$ of \emph{Hamiltonian symplectomorphisms} of $(B,\omega)$. It
follows from Theorem~\ref{thm:hamvf} that this Lie algebra can be
naturally identified with the vector space $C^\infty (B) / \R$, 
i.e. smooth functions on $B$ modulo constants, equipped with a bracket 
$\left\{\cdot,\cdot\right\}$ known as the Poisson bracket:
\[
\left\{f,g\right\} \equiv \omega(X_f, X_g)\,.
\]
Note that when $H^1(B, \R) = 0$ we have that ${\mathcal X}_H(B,\omega) =
{\mathcal X}(B,\omega)$. 
\end{remark}

\subsection{Hamiltonian Torus Actions} 

Let $(B,\omega)$ be a symplectic manifold equipped with a symplectic
action of 
\[
{\mathbb T}^m \equiv \R^m/2\pi\Z^m \equiv \R/2\pi\Z \times \cdots \times
\R/2\pi\Z \equiv {\mathbb S}^1 \times \cdots \times {\mathbb S}^1\,,
\]
i.e. with a homomorphism ${\mathbb T}^m \to \Diff(B,\omega)$. Let $X_1, \ldots,
X_m \in {\mathcal X}(B)$ be the vector fields generating the action of each
individual ${\mathbb S}^1$-factor. Then, since the action is symplectic, we have
that 
\[
{\mathcal L}_{X_k} \omega = 0 \Leftrightarrow X_k \lrcorner\,d\omega + d(X_k
\lrcorner\, \omega) = 0 \Leftrightarrow d(X_k
\lrcorner\, \omega) = 0\,,
\]
i.e.
\[
X_k \in {\mathcal X} (B,\omega)\,,\ \forall\, k\in\left\{1,\ldots,m\right\}\,.
\]
\begin{defn} \label{def:hamabaction}
A symplectic ${\mathbb T}^m$-action on a symplectic manifold $(B,\omega)$ is said
to be \emph{Hamiltonian} if for every $k\in\left\{1,\ldots,m\right\}$ there
exists a function $h_k:B\to\R$ such that $X_k\lrcorner\,\omega = d h_k$,
i.e. $X_k\equiv X_{h_k}\in{\mathcal X}_H (B,\omega)$ 
is the Hamiltonian vector field of $h_k$. In
this case, the map $\mu:B\to\R^m$ defined by
\[
\mu(p) = (h_1(p), \ldots, h_m(p))\,,\ \forall\, p\in B\,,
\]
is called a \emph{moment map} for the action.
\end{defn}
\begin{remark}
Suppose $\mu:B\to\R^m$ is a moment map for a Hamiltonian ${\mathbb T}^m$-action on
$(B,\omega)$. Then $\mu + c$, for any given constant $c\in\R^m$, is also
a moment map for that same action.
\end{remark}
\begin{remark} \label{rmk:isotorb}
The orbits of a Hamiltonian ${\mathbb T}^m$-action on a symplectic manifold
$(B, \omega)$ are always \emph{isotropic}, i.e.
\[
\omega|_{\text{orbit}} \equiv 0\,.
\]
In fact, the tangent space to an orbit is generated by the Hamiltonian
vector fields $X_{h_k}$, $k\in\{1,\ldots,m\}$. Using
Theorem~\ref{thm:hamvf} and the fact that the torus ${\mathbb T}^m$ is abelian,
we have that
\[
X_{\omega(X_{h_k}, X_{h_l})} = - \left[X_{h_k},X_{h_l}\right] \equiv 0 
\Rightarrow \omega(X_{h_k}, X_{h_l}) \equiv \text{const.}\,,\ \forall\,
k,l\in \{1,\ldots,m\}\,.
\]
Since ${\mathbb T}^m$ is compact, there is for each $k\in\{1,\ldots,m\}$ and on
each ${\mathbb T}^m$-orbit a point $p_k$ where the function $h_k
|_{\text{orbit}}$ attains its maximum. Then
\[
\omega(X_{h_k}, X_{h_l}) = (d h_k)_{p_k} (X_{h_l}) = 0\,.
\]
Hence, the above constant is actually zero and each ${\mathbb T}^m$-orbit is
indeed isotropic. This fact will be used below, in the proof of
Proposition~\ref{prop:dimtorus}.
\end{remark}
\begin{example}\label{ex:hamR2n}
Consider {$(\R^{2n}, \omega_{\rm st})$}, where
\[
\omega_{\rm st} = d u \wedge d v := \sum_{j=1}^n d u_j \wedge d v_j
\]
as in Example~\ref{ex:R2n}, and its usual identification with $\C^n$ given 
by 
\begin{equation} \label{eq:R2n-Cn}
z_j = u_j + i v_j\,,\ j=1,\ldots, n\,.
\end{equation}
The standard ${\mathbb T}^n$-action {$\tau_{\rm st}$} on $\R^{2n}$, given by
\[
(y_1, \ldots, y_n) \cdot (z_1,\ldots, z_n) = 
(e^{-i y_1}z_1, \ldots, e^{-i y_n}z_n)\,,
\]
is Hamiltonian, with moment map ${\mu_{\rm st}}:\R^{2n} \to \R^n$ given by
\[
\mu_{\rm st} (u_1, \ldots, u_n, v_1, \ldots, v_n) = 
\frac{1}{2} (u_1^2 + v_1^2, \ldots, u_n^2 + v_n^2)\,.
\]
\end{example}
\begin{example} \label{ex:hamPn}
Consider projective space {$(\CP^n, \omega_{\rm FS})$}, with homogeneous 
coordinates $[z_0; z_1; \ldots; z_n]$. 

The ${\mathbb T}^n$-action {$\tau_{\rm FS}$} on $\CP^{n}$ given by
\[
(y_1, \ldots, y_n) \cdot [z_0; z_1;\ldots; z_n] = 
[z_0; e^{-i y_1}z_1; \ldots; e^{-i y_n}z_n]\,,
\]
is Hamiltonian, with moment map ${\mu_{\rm FS}}:\CP^{n} \to \R^n$ given by
\[
\mu_{\rm FS} [z_0; z_1; \ldots; z_n] = 
\frac{1}{\|z\|^2} (\|z_1\|^2, \ldots, \|z_n\|^2)\,.
\] 

Note that the image of $\mu_{\rm FS}$ is the convex hull of the images 
of the $n+1$ fixed points of the action, i.e. the standard simplex in $\R^n$.
\end{example}

Atiyah~\cite{At} and Guillemin-Sternberg~\cite{GS} 
proved in 1982 the following Convexity Theorem.

\begin{thm} \label{thm:ags-convex}
Let $(B,\omega)$ be a compact, connected, symplectic manifold, equipped
with a Hamiltonian ${\mathbb T}^m$-action with moment map $\mu:B\to\R^m$. Then
\begin{itemize}
\item[(i)] the level sets $\mu^{-1}(\lambda)$ of the moment map are
  connected (for any $\lambda\in\R^m$); 
\item[(ii)] the image $\mu(B)\subset\R^m$ of the moment map is the
  convex hull of the images of the fixed points of the action.
\end{itemize}
\end{thm}

\subsection{Toric Symplectic Manifolds}

The following proposition motivates the definition of a toric symplectic
manifold.

\begin{prop} \label{prop:dimtorus}
If a symplectic manifold $(B,\omega)$ has an effective Hamiltonian
${\mathbb T}^m$-action, then $m\leq(\dim B)/2$.
\end{prop}

\begin{proof}
\[
\begin{array}{rcl}
\text{Effective action} & \Rightarrow & \text{there exist
  $m$-dimensional orbits.} \\
\text{Hamiltonian ${\mathbb T}^m$-action} & \Rightarrow & \text{orbits are
  isotropic (see Remark~\ref{rmk:isotorb}).} \\
\text{Linear Algebra} & \Rightarrow & \dim(\text{isotropic orbit})
  \leq \frac{1}{2}\dim B\,.
\end{array}
\]
\end{proof}

\begin{defn} \label{def:toric}
A \emph{toric symplectic manifold} is a connected symplectic manifold 
$(B^{2n}, \omega)$, equipped with an effective Hamiltonian action 
of the $n$-torus,
\[
\tau:{\mathbb T}^n \cong \R^n / 2\pi \Z^n \hookrightarrow \Ham (B,\omega)\,,
\]
such that the corresponding \emph{moment map} 
\[
\mu : B \to \R^n\,,
\]
well-defined up to a constant, is proper onto its convex image 
$P=\mu(B)\subset \R^n$.
\end{defn}

\begin{remark} \label{rem:toric}
The requirement that the moment map be ``proper onto its convex image'', 
something that is automatic for compact manifolds, makes the theory
of non-compact toric symplectic manifolds analogous to the compact one
(see~\cite{kl}).
\end{remark}

\begin{example} \label{ex:hamR2ntor}
$(\R^{2n}, \omega_{\rm st})$, equipped with the standard Hamiltonian ${\mathbb T}^n$-action
described in Example~\ref{ex:hamR2n}, is a non-compact toric symplectic manifold.
\end{example}
\begin{example} \label{ex:hamPntor}
$(\CP^n, \omega_{\rm FS})$, equipped with the Hamiltonian ${\mathbb T}^n$-action described
in Example~\ref{ex:hamPn}, is a compact toric symplectic manifold. 
\end{example}

\subsection{Classification Theorem and Action-Angle Coordinates}

Any toric symplectic manifold has an associated convex set, the image of
the moment map of the torus action. The convex sets that arise 
in this way are characterized in the following definition.

\begin{defn} \label{def:image}
A convex polyhedral set $P$ in $\R^n$ is called \emph{simple} and
\emph{integral} if:
\begin{itemize}
\item[(1)] there are $n$ edges meeting at each vertex $p$;
\item[(2)] the edges meeting at the vertex $p$ are rational, 
i.e. each edge is of the form $p + tv_i,\ 0\leq t\leq \infty,\ 
{\rm where}\ v_i\in\Z^n$;
\item[(3)] the $v_1, \ldots, v_n$ in (2) can be chosen to be a 
$\Z$-basis of the lattice $\Z^n$.
\end{itemize}
A \emph{facet} is a face of $P$ of codimension one.

A \emph{Delzant set} is a simple and integral convex polyhedral set 
$P\subset \R^n$. A \emph{Delzant polytope} is a \emph{compact} Delzant set.

Two Delzant sets are \emph{isomorphic} if one can be mapped to the other
by a translation.
\end{defn}

In 1988 Delzant~\cite{De} showed that any Delzant polytope determines a 
unique compact toric symplectic manifold. More precisely, if two compact toric 
symplectic manifolds have the same Delzant polytope, then there exists an 
equivariant symplectomorphism between them. This result can be generalized to the 
possibly non-compact setting of Definition~\ref{def:toric} (see~\cite{kl}).

\begin{thm} \label{thm:genDel}
Let $(B,\omega,\tau)$ be a toric symplectic manifold, with
moment map $\mu : B \to \R^n$. Then $P\equiv \mu(B)$ is a
Delzant set. 

Two toric symplectic manifolds are equivariant symplectomorphic
(with respect to a fixed torus acting on both) if and only if their
associated Delzant sets are isomorphic. Moreover, every Delzant set arises from some 
toric symplectic manifold.
\end{thm}

\begin{remark} \label{rem:labels}
One can use the work of Lerman and Tolman~\cite{LeTo} to generalize Theorem~\ref{thm:genDel}
to orbifolds. The outcome is a classification of symplectic toric orbifolds via
\emph{labeled Delzant sets}, i.e. convex polyhedral sets, as in Definition~\ref{def:image}, with ``$\Z$-basis'' in (3) replaced by ``$\Q$-basis'' and with a positive integer label
attached to each facet. 

Each facet $F$ of a labeled Delzant set $P\subset \R^n$ determines a unique lattice vector $\nu_F\in\Z^n$: the primitive inward pointing normal lattice vector. A convenient way of thinking about a positive integer label $m_F\in\N$ associated to $F$ is by dropping the primitive 
requirement from this lattice vector: consider $m_F\nu_F$ instead of $\nu_F$.

In other words, a labeled Delzant set can be defined as a rational simple polyhedral set 
$P\subset \R^n$ with an inward pointing normal lattice vector associated to each of its facets. 
\end{remark}

The proof gives an explicit symplectic reduction construction of a canonical model for each  
toric symplectic manifold, i.e. it associates to each Delzant set $P$ 
an explicit toric symplectic manifold $(B_P,\omega_P,\tau_P)$ with moment map $\mu_P:B_P\to P$. 
One can use these canonical models to derive general properties of toric symplectic manifolds. 
For example, let $\breve{P}$ denote the interior of $P$, and consider 
$\breve{B}_P\subset B_P$ defined by $\breve{B}_P = \mu_P^{-1}
(\breve{P})$. One easily checks that $\breve{B}_P$ is an open dense subset of $B_P$, consisting 
of all the points where the ${\mathbb T}^n$-action is free. It can be described as
\[
\breve{B}_P\cong \breve{P}\times {\mathbb T}^n =
\left\{ (x,y): x\in\breve{P}\subset\R^n\,,\ 
y\in\R^n/2\pi\Z^n\right\}\,,
\]
where $(x,y)$ are symplectic \emph{action-angle} coordinates for
$\omega_P$, i.e.
$$
\omega_P|_{\breve{B}} = d x\wedge d y = \sum_{j=1}^n d x_j \wedge d y_j\ .
$$
Hence, one has a global equivariant Darboux's Theorem in this toric context.
Note that in these action-angle coordinates the moment map is simply given
by
\[
\mu_P (x,y) = x\,,
\]
i.e. projection in the action coordinates.

\section{Toric K\"ahler Metrics} \label{s:toricKmetrics}

\subsection{Linear Compatible Complex Structures}

\begin{defn} \label{def:lccs}
A \emph{compatible complex structure} on a symplectic vector space
$(V,\omega)$ is a complex structure $J$ on $V$, i.e. 
$J \in \End (V))$ with $J^2 = - \Id$, such that
\[
\langle\cdot,\cdot\rangle_J := \omega(\cdot,J\cdot)
\]
is an inner product on $V$. This is equivalent to 
\[
\omega(J\cdot,J\cdot) = \omega(\cdot,\cdot)
\quad\text{and}\quad
\omega(v,Jv)>0\,,\ \forall\, 0\ne v\in V\,.
\] 
The set of all compatible complex structures on a symplectic vector
space $(V,\omega)$ will be denoted by ${\mathcal J}(V,\omega)$.
\end{defn}

The symplectic linear group $Sp(V,\omega)$ acts on ${\mathcal J} (V,\omega)$ by 
conjugation:
\begin{align}
Sp(V,\omega)\times {\mathcal J}(V,\omega) & \to {\mathcal J}(V,\omega) \notag \\
(\Phi, J) & \mapsto \Phi J \Phi^{-1} \notag
\end{align}
This action can be easily seen to be transitive and, if we fix 
$J_0 \in {\mathcal J}(V,\omega)$ and corresponding inner product 
$\langle\cdot,\cdot\rangle_0$,
we have that
\[
{\mathcal J} (V,\omega) = Sp (V,\omega) / 
U (V,\omega, \langle\cdot,\cdot\rangle_0) \,,
\]
where $U (V,\omega, \langle\cdot,\cdot\rangle_0) = Sp (V,\omega) 
\cap O (V, \langle\cdot,\cdot\rangle_0)$ is the unitary group. 
${\mathcal J}(V,\omega)$ is a symmetric space and admits a beautiful 
explicit description due to C.L.~Siegel~\cite{Sie}.

\begin{defn} \label{def:suhs}
The \emph{Siegel upper half space} ${\mathcal S}_n$ is the open 
\emph{contractible} subset of the complex vector space of complex 
symmetric matrices defined by
\begin{align}
{\mathcal S}_n := \{Z = R + i S : & \ \text{$R$ and $S$ are real 
\emph{symmetric} $(n\times n)$ matrices,} \notag \\
& \ \text{with \emph{$S$ positive definite}}\}\,.\notag
\end{align}
\end{defn}

Choose a symplectic basis for $(V, \omega)$, i.e. an isomorphism 
$(V,\omega) \cong (\R^{2n}, \omega_{\rm st})$, and identify the symplectic 
linear group $Sp(V,\omega)$ with the matrix group $Sp (2n,\R)$ consisting of 
$(2n\times 2n)$ real matrices $\Phi$ such that
\[
\Phi^t  \cdot \omega_0 \cdot \Phi = \omega_0\,,
\]
where
\[
\omega_0 =
\begin{bmatrix}
0 & \vdots & \Id \\
\hdotsfor{3} \\
-\Id & \vdots & 0
\end{bmatrix} 
\]
is the matrix form of $\omega_{\rm st}$ written in $(n\times n)$ blocks.
We will also write any $\Phi\in Sp(2n,\R)$ in $(n\times n)$ blocks:
\[
\Phi =
\begin{bmatrix}
A & \vdots & B \\
\hdotsfor{3} \\
C & \vdots & D
\end{bmatrix} \,. 
\]
The following proposition is proved in~\cite{Sie}.

\begin{prop}
$Sp(2n,\R)$ acts on ${\mathcal S}_n$ by linear fractional transformations:
\begin{align}
Sp(2n,\R)\times {\mathcal S}_n & \to {\mathcal S}_n \notag \\
(\Phi, Z) & \mapsto \Phi (Z) := (AZ+B)\cdot (CZ+D)^{-1}\,. \notag
\end{align}
This action is transitive and the isotropy group of $i\Id\in{\mathcal S}_n$ is
$U(n)\subset Sp(2n,\R)$:
\[
U(n) = Sp(2n,\R)\cap O(n) = 
\left\{ \Phi\in Sp(2n,\R)\,:\ \Phi^t \cdot \Phi = \Id \right\}\,.
\]
Hence,
\[
{\mathcal S}_n \cong Sp(2n,\R) / U(n)\,.
\]
\end{prop}

Given $Z = R + i S \in {\mathcal S}_n$, define
\[
\Phi_Z :=
\begin{bmatrix}
S^{1/2} & \vdots & R S^{-1/2} \\
\hdotsfor{3} \\
0 & \vdots & S^{-1/2}
\end{bmatrix} 
\in Sp(2n,\R)\,.
\]
Under the action of $Sp(2n,\R)$ on ${\mathcal S}_n$, we have that
\[
\Phi_Z (i\Id) = Z\,.
\]
Let $J_0 \in {\mathcal J}(\R^{2n},\omega_{\rm st})$
be given by
\begin{equation} \label{eq:J0}
J_0 =
\begin{bmatrix}
0 & \vdots & -\Id \\
\hdotsfor{3} \\
\Id & \vdots & 0
\end{bmatrix} \,.
\end{equation}
For each $Z\in{\mathcal S}_n$, define 
$J_Z \in {\mathcal J}(\R^{2n},\omega_{\rm st})$ by
\[
J_Z := (J_0\cdot\Phi_Z)\cdot J_0\cdot (J_0\cdot\Phi_Z)^{-1} =
\begin{bmatrix}
-S^{-1}R & \vdots & - S^{-1}\\
\hdotsfor{3} \\
R S^{-1} R + S & \vdots & R S^{-1}
\end{bmatrix} 
\,.
\]
This defines a bijection
\begin{align}
{\mathcal S}_n & \to {\mathcal J}(\R^{2n},\omega_{\rm st}) \notag \\
Z & \mapsto J_Z \notag
\end{align}
which, up to $J_0$-conjugation, is equivariant with respect to the
$Sp(2n,\R)$-action on both spaces. More precisely, 
if $\Phi\in Sp(2n,\R)$ then
\[
\Phi \cdot J_Z \cdot \Phi^{-1} = J_{Z'} \Leftrightarrow
Z' = (J_0^{-1} \cdot \Phi \cdot J_0)(Z) \,.
\]
In particular, ${\mathcal J}(\R^{2n},\omega_{\rm st})$ is a contractible space and, 
for any symmetric $(n\times n)$ real matrix $U$, we have 
that
\begin{equation} \label{eq:conj}
\Phi = 
\begin{bmatrix}
I & \vdots & 0\\
\hdotsfor{3} \\
U & \vdots & I
\end{bmatrix} 
\in Sp(2n,\R)
\quad\text{and}\quad
\Phi \cdot J_Z \cdot \Phi^{-1} = J_{(Z-U)}\,.
\end{equation}
This will be relevant below.

\subsection{Toric Compatible Complex Structures}

\begin{defn} \label{def:cacs}
A \emph{compatible almost complex structure} on a symplectic manifold
$(B,\omega)$ is an almost complex structure $J$ on $B$, i.e. 
$J \in \Gamma (\End (TB))$ with $J^2 = - \Id$, such that
\[
\langle\cdot,\cdot\rangle_J := \omega(\cdot,J\cdot)
\]
is a Riemannian metric on $B$. This is equivalent to $\omega(J\cdot,
J\cdot) = \omega(\cdot,\cdot)$ and $\omega(X,JX)>0\,,\ \forall\, 0\ne X\in
TB$. 

The space of all compatible almost complex structures on a symplectic
manifold $(B,\omega)$ will be denoted by ${\mathcal J}(B,\omega)$.
\end{defn}
\begin{remark} \ 
\begin{itemize}
\item[(i)] The fact that ${\mathcal J}(\R^{2n},\omega_{\rm st})$ is contractible 
implies that ${\mathcal J}(B,\omega)$ is non-empty, infinite-dimensional and contractible, for any symplectic manifold $(B,\omega)$. 
\item[(ii)] A \emph{K\"ahler manifold} is a symplectic manifold $(B,\omega)$
  with an integrable compatible complex structure $J$, i.e. one that is
  locally isomorphic to the standard complex structure $J_0$ on $\R^{2n}$.
  Note that~(\ref{eq:R2n-Cn}) gives the standard isomorphism
  $(\R^{2n}, J_0) \cong \C^n$.
\item[(iii)] The space of integrable compatible complex structures on a 
symplectic manifold $(B,\omega)$ will be denoted by 
${\mathcal I}(B,\omega)\subset {\mathcal J}(B,\omega)$.
\item[(iv)] In general, ${\mathcal I}(B,\omega)$ can be empty.
\end{itemize}
\end{remark}

\begin{defn} \label{def:tccs}
A \emph{toric compatible complex structure} on a toric symplectic manifold 
$(B^{2n}, \omega, \tau)$ is a ${\mathbb T}^n$-invariant $J\in{\mathcal I}(B,\omega)\subset{\mathcal J}(B,\omega)$.
The space of all such will be denoted by ${\mathcal I}^{{\mathbb T}^n}(B,\omega)\subset{\mathcal J}^{{\mathbb T}^n}(B,\omega)$.
\end{defn}

\begin{remark}
It follows from the classification in Theorem~\ref{thm:genDel}, more precisely 
from the explicit symplectic reduction construction of the canonical model for 
any compact toric symplectic manifold $(B^{2n}, \omega, \tau)$, that 
${\mathcal I}^{{\mathbb T}^n}(B,\omega)$ is always non-empty.
\end{remark}

\subsection{Local Form of Toric Compatible Complex Structures}

It follows from the above bijection between ${\mathcal J}(\R^{2n},\omega_{\rm st})$ 
and the Siegel upper half space ${\mathcal S}_n$ that any 
$J\in {\mathcal J}^{{\mathbb T}^n}(\breve{B}, \omega|_{\breve{B}})$ can be written 
in action-angle coordinates $(x,y)$ on 
$\breve{B}\cong\breve{P}\times{\mathbb T}^n$ as
\[
J = 
\begin{bmatrix}
- S^{-1} R & \vdots & -S^{-1} \\
\hdotsfor{3} \\
R S^{-1}R + S & \vdots & R S^{-1}
\end{bmatrix}    
\]
where $R = R(x)$ and $S=S(x)$ are real symmetric $(n\times n)$ matrices, 
with $S$ positive definite.

For \emph{integrable} toric compatible complex structures we have that:
\begin{align}
& J \in {\mathcal I}^{{\mathbb T}^n} \subset {\mathcal J}^{{\mathbb T}^n} 
\Leftrightarrow \frac{{\partial}Z_{ij}}{{\partial}x_k} = \frac{{\partial}Z_{ik}}{{\partial}x_j} \notag \\
& \notag \\
\Leftrightarrow & \exists\ f:\breve{P} \to \C\,,\ f(x) = r(x) + i s(x)\,,\ 
\text{such that} \notag \\
& Z_{ij} = \frac{{\partial}^2 f}{{\partial}x_i {\partial}x_j} = 
\frac{{\partial}^2 r}{{\partial}x_i {\partial}x_j} + i \frac{\partial^2 s}{{\partial}x_i {\partial}x_j} =
R_{ij} + i S_{ij} \,.\notag
\end{align}

Any real function $h:\breve{P}\to\R$ is the Hamiltonian of a $1$-parameter family 
\[
\phi_t:\breve{B}\to\breve{B}
\] 
of ${\mathbb T}^n$-equivariant symplectomorphisms. These are
given in action-angle coordinates $(x,y)$ on 
$\breve{B} \cong \breve{P}\times{\mathbb T}^n$ by
\[
\phi_t(x,y) = (x, y - t \frac{{\partial}h}{{\partial}x})\,.
\]
Hence, it follows from~(\ref{eq:conj}) that the natural action of 
$\phi_t$ on ${\mathcal J}^{{\mathbb T}^n}$, given by
\[
\phi_t \cdot J = (d\phi_t) \circ J \circ (d\phi_t)^{-1}\,,
\]
corresponds in the Siegel upper half space parametrization to
\[
\phi_t \cdot (Z = R+i S) = (R + tH) + i S\,,
\]
where
\[
H = (h_{ij}) = \left(\frac{{\partial}^2 h}{{\partial}x_i {\partial}x_j} \right)\,.
\]
This implies that, for any integrable $J\in{\mathcal I}^{{\mathbb T}^n}$, there exist 
action-angle coordinates $(x,y)$ on $\breve{B} \cong \breve{P}\times{\mathbb T}^n$ 
such that $R\equiv 0$, i.e. such that
\[
J = 
\begin{bmatrix}
0 & \vdots & -S^{-1} \\
\hdotsfor{3} \\
S & \vdots & 0
\end{bmatrix}    
\]
with
\[
S = S(x) = \left(s_{ij}(x)\right) = 
\left(\frac{{\partial}^2 s}{{\partial}x_i {\partial}x_j}\right)
\]
for some
\[
\text{real potential function} \quad s:\breve{P}\to\R\,.
\]
Holomorphic coordinates for $J$ are given in this case by
\begin{equation} \label{eq:holcoords}
z(x,y) = u(x,y) + i v(x,y) = \frac{{\partial}s}{{\partial}x}(x)
+ i y\,.
\end{equation}

The corresponding Riemannian (K\"ahler) metric 
\[
\langle\cdot, \cdot\rangle_J := \omega (\cdot , J\cdot)
\]
on $\breve{B} \cong \breve{P}\times{\mathbb T}^n$ can the be written in matrix form as
\begin{equation} \label{metricG}
\omega_0 \cdot J = 
\begin{bmatrix}
0 & \vdots & \Id \\
\hdotsfor{3} \\
-\Id & \vdots & 0
\end{bmatrix} 
\cdot
\begin{bmatrix}
0 & \vdots & -S^{-1} \\
\hdotsfor{3} \\
S & \vdots & 0
\end{bmatrix}    
=
\begin{bmatrix}
S & \vdots & 0 \\
\hdotsfor{3} \\
0 & \vdots & S^{-1}
\end{bmatrix}    
\end{equation}
with
\[
S = \left(\frac{{\partial}^2 s}{{\partial}x_i {\partial}x_j}\right)\,.
\]

\begin{defn} \label{def:symppot}
We will call such a potential function
\[ 
s : \breve{P} \to \R 
\] 
the \emph{symplectic potential} of both the complex structure 
$J$ and the metric $\langle\cdot, \cdot\rangle_J$
\end{defn}

\begin{remark}
This particular way to arrive at the above local form for any 
$J\in{\mathcal I}^{{\mathbb T}^n}$ is due to Donaldson~\cite{D4}, 
and illustrates a small part of his formal general framework 
for the action of the symplectomorphism group of a symplectic 
manifold on its space of compatible complex structures
(cf.~\cite{D1}).
\end{remark}

\begin{example} \label{ex:symppotR2n}
Consider the standard linear complex structure 
$J_0\in{\mathcal I}^{{\mathbb T}^n}(\R^{2n}, \omega_{\rm st})$ 
given by~(\ref{eq:J0}). In action-angle coordinates $(x,y)$ on 
\[
\breve{\R}^{2n} = (\R^2\setminus\{(0,0)\})^n \cong 
(\R^+)^n \times {\mathbb T}^n  = \breve{P}\times{\mathbb T}^n\,,
\]
its symplectic potential is given by
\begin{align}
s:\breve{P} = (\R^+)^n & \longrightarrow \R \notag \\
x = (x_1, \ldots, x_n) & \longmapsto 
s(x) = \frac{1}{2} \sum_{i=1}^n x_i \log (x_i) \,. \notag
\end{align}

Hence, in these action-angle coordinates, the standard complex 
structure has the matrix form
\[
J_0 = 
\begin{bmatrix}
0 & \vdots & \diag(-2x_i) \\
\hdotsfor{3} \\
\diag(1/2x_i) & \vdots & 0
\end{bmatrix}    
\]
while the standard flat Euclidean metric becomes
\[
\begin{bmatrix}
\diag(1/2x_i) & \vdots & 0 \\
\hdotsfor{3} \\
0 & \vdots & \diag(2x_i)
\end{bmatrix} \,.
\]
\end{example}

\subsection{Symplectic Potentials for Compact Toric Symplectic Manifolds}

The proof of Theorem~\ref{thm:genDel} associates
to each Delzant set $P\subset\R^n$, via an explicit symplectic reduction 
construction, a canonical K\"ahler toric manifold 
\[
(B_P^{2n}, \omega_P, \tau_P, \mu_P, J_P)
\quad\text{such that}\quad\mu_P (B_P) = P\,.
\]
In~\cite{Gui1} Guillemin gave an explicit formula for the symplectic potential
of this canonical K\"ahler metric. To write it down one just needs some simple
combinatorial data that can be easily obtained directly from the polytope $P$.

Let $F_i$ denote the $i$-th facet (codimension-$1$ face) of the polytope $P$. 
The affine defining function of $F_i$ is the function
\begin{align}
\ell_i : \R^n & \longrightarrow \R \notag \\
x & \longmapsto \ell_i (x) = \langle x, \nu_i \rangle - \lambda_i\,, \notag
\end{align}
where $\nu_i\in\Z^n$ is a primitive inward pointing normal to $F_i$ and $\lambda_i\in\R$ 
is such that $\ell_i|_{F_i} \equiv 0$. Note that $\ell_i|_{\breve{P}} > 0$.

\begin{theorem} \label{thm:Gui1}
In appropriate action-angle coordinates $(x,y)$, the canonical symplectic 
potential $s_P:\breve{P}\to\R$ for $J_P |_{\breve{P}}$ is given by
\[
s_P (x) = \frac{1}{2} \sum_{i=1}^d \ell_i (x) \log \ell_i (x)\,,
\]
where $d$ is the number of facets of $P$.
\end{theorem}

\begin{example} \label{ex:symppotR2n-2}
The symplectic potential presented in Example~\ref{ex:symppotR2n} for
the standard flat Euclidean metric on $\R^{2n}$ is the canonical
symplectic potential of the corresponding Delzant set
$P = (\R^+_0)^n \subset \R^n$.
\end{example}

\begin{example} \label{ex:symppotPn}
For projective space $\CP^n$ the polytope $P\subset\R^n$ can be taken to be the 
standard simplex, with defining affine functions
\[
\ell_i (x) = x_i\,,\ i=1,\ldots,n\,, 
\quad\text{and}\quad
\ell_{n+1}(x) = 1 - r\,,
\]
where $r = \sum_i x_i$. 

The canonical symplectic potential $s_P:\breve{P}\to\R$, given by
\[
s_P (x) = \frac{1}{2} \sum_{i=1}^n x_i \log x_i + 
\frac{1}{2}(1-r) \log (1-r)\,,
\]
defines the standard complex structure $J_{FS}$ and Fubini-Study metric
on $\CP^n$.
\end{example}

Theorem~\ref{thm:Abr2} below provides the symplectic version of the 
${\partial}{\overline{\partial}}$-lemma in this toric context, 
characterizing the symplectic potentials that correspond to toric 
compatible complex structures on a toric symplectic manifold.
It is an immediate extension to our possibly non-compact setting
of the compact version proved in~\cite{Abr2}. To properly state it
we need the following definition.

\begin{defn} \label{def:Jcomplete}
Let $(B,\omega,\tau)$ be a symplectic toric manifold and denote by 
\[
Y_1,\ldots, Y_n \in{\mathcal X}_H(B, \omega)
\] 
the Hamiltonian vector fields generating the torus action. A toric 
compatible complex structure
$J\in{\mathcal I}^{\mathcal T}(B,\omega)$ is said to be \emph{complete} 
if the $J$-holomorphic vector fields 
\[
JY_1,\ldots, JY_n \in{\mathcal X}(B)
\] 
are complete. The space of all complete toric compatible complex
structures on $(B,\omega,\tau)$ will be denoted by 
${\mathcal I}_c^{\mathcal T}(B,\omega)$.
\end{defn}

\begin{theorem} \label{thm:Abr2}
Let $J$ be any complete compatible toric complex structure on the 
symplectic toric manifold $(B_P,\omega_P,\tau_P)$. Then, in suitable
action-angle $(x,y)$-coordinates on $\breve{B}_P\cong\breve{P}
\times {\mathbb T}^n$, $J$ is given by a symplectic potential
$s\in C^\infty(\breve{P})$ of the form
\[
s(x) = s_P (x) + h(x)\,,
\]
where $s_P$ is given by Theorem~\ref{thm:Gui1}, $h$ is smooth on the whole
$P$, and the matrix $S=\Hess_x(s)$ is positive definite on $\breve{P}$
and has determinant of the form
\[
\Det(S) = \left(\delta \prod_{r=1}^d \ell_r \right)^{-1}\,,
\]
with $\delta$ being a smooth and strictly positive function on the whole $P$.

Conversely, any such potential $s$ determines a (not necessarily complete)
complex structure on $\breve{B}_P\cong\breve{P}\times {\mathbb T}^n$, that 
extends uniquely to a well-defined compatible toric complex structure $J$ 
on the toric symplectic manifold $(B_P,\omega_P,\tau_P)$.
\end{theorem}

\begin{remark} \label{rem:Abr3}
If one takes into account Remark~\ref{rem:labels}, the word ``manifold'' 
can be replaced by ``orbifold'' in Theorems~\ref{thm:Gui1} 
and~\ref{thm:Abr2} (see~\cite{Abr3}). 
\end{remark}

\begin{remark}
There is no immediate relation between completeness of a toric compatible 
complex structure and completeness of the associated toric K\"ahler metric.
See Remark~\ref{rem:complete}.
\end{remark}

\subsection{Scalar Curvature} 
\label{ssec:scalcurveq}

We now recall from~\cite{Abr1} a particular formula for the scalar curvature 
in action-angle $(x,y)$-coordinates.
A K\"ahler metric of the form~(\ref{metricG}) has scalar curvature $Sc$ given 
by\footnote[1]{The normalization for the value of the scalar curvature we are 
using here differs from the one used in~\cite{Abr1,Abr2} by a factor of $1/2$.}
\[
Sc = - \sum_{j,k} \frac{\partial}{{\partial}x_j}
\left( g^{jk}\, \frac{{\partial}\log \Det(S)}{{\partial}x_k} \right)\,,
\]
which after some algebraic manipulations becomes the more compact
\begin{equation} \label{scalarsymp}
Sc = - \sum_{j,k} \frac{{\partial}^2 s^{jk}}{{\partial}x_j {\partial}x_k}\,, 
\end{equation}
where the $s^{jk},\ 1\leq j,k\leq n$, are the entries of the inverse 
of the matrix $S = \Hess_x (s)$, $s\equiv$ symplectic potential.
See~\cite{D1} for an appropriate interpretation of this formula for
the scalar curvature. 

\subsection{Symplectic Potentials and Affine Transformations}

Because the Delzant set $P\subset\R^n$ of a symplectic toric manifold
is only really well defined up to translations (i.e. additions 
of constants to the moment map) and $SL(n,\Z)$ transformations (i.e.
changes of basis of the torus ${\mathbb T}^n = \R^n / 2\pi\Z^n$), 
symplectic potentials should transform naturally under these type of maps. 
While the effect of translations is trivial to analyse, the effect of
$SL(n,\Z)$ transformations is more interesting. In fact:
\begin{center}
symplectic potentials transform quite naturally

under any $GL(n,\R)$ linear transformation.
\end{center}

Let $T\in GL(n,\R)$ and consider the linear symplectic change
of action-angle coordinates
\[
x := T^{-1}x'\quad\text{and}\quad y:=T^t y'\,.
\]
Then 
\[
P' = \bigcap_{a=1}^d \{x'\in\R^n\,:\ 
\ell'_a (x') := \langle x', \nu'_a \rangle + \lambda'_a \geq 0\} 
\]
becomes
\[
P:= T^{-1}(P') = \bigcap_{a=1}^d \{x\in\R^n\,:\ \ell_a (x) := 
\langle x, \nu_a \rangle + \lambda_a \geq 0\}
\]
with 
\[
\nu_a = T^t \nu'_a\quad\text{and}\quad \lambda_a = \lambda'_a\,,
\]
and symplectic potentials transform by
\[
s = s'\circ T \quad\text{(in particular, $s_{P} = s_P' \circ T)$.}
\]
The corresponding Hessians are related by
\[
S = T^t (S'\circ T) T
\]
and
\[
Sc = Sc' \circ T\,.
\]

\begin{example} \label{ex:hirzebruch}
Figure~\ref{fig:hirzebruch} illustrates two equivalent descriptions 
of a toric symplectic rational ruled $4$-manifold or, equivalently,
of a Hirzebruch surface 
\[
H^2_m := \CP ({\mathcal O}(-m)\oplus\C) \to \CP^{1}\,, \ m\in\N\,.
\] 
The linear map $T\in GL(2,\R)$ relating the two is given by
\[
T = 
\begin{bmatrix}
m & -1 \\
0 & \ 1
\end{bmatrix}    
\]
The inward pointing normal that should be considered for each facet is
specified. The polytope on the right is a standard Delzant polytope for
the Hirzebruch surface $H^2_m$. The polytope on the left is very useful 
for the K\"ahler metric constructions of section~\ref{s:calabi} and
was implicitly used by Calabi in~\cite{C2}.
\end{example}

\begin{figure}
      \centering
      \includegraphics[viewport = 50 0 250 100, scale=1.0]{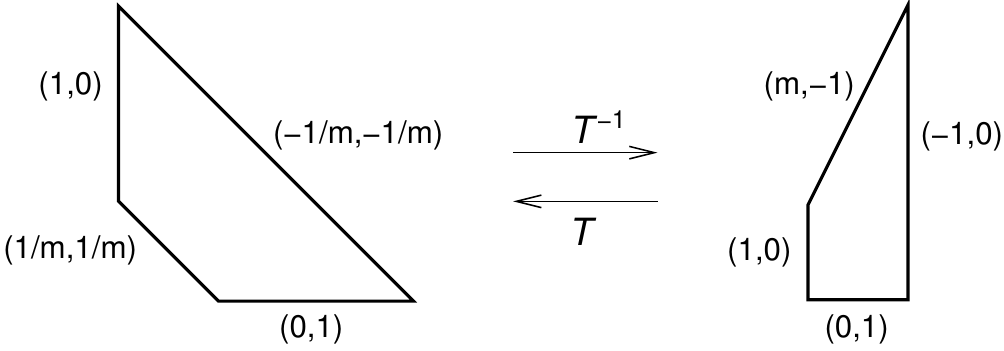}
      \caption{Hirzebruch surfaces.}
      \label{fig:hirzebruch}
\end{figure}

\section{Toric Constant Curvature Metrics in Real Dimension 2}
\label{s:2d-metrics}

In real dimension $2$ any orientable Riemannian manifold is K\"ahler,
since its area form is a symplectic form and oriented rotation by $\pi/2$
on each tangent plane is a compatible complex structure. In this section
we write down the symplectic potentials that give rise to toric constant
(scalar) curvature metrics in real dimension $2$ and identify the 
underlying toric symplectic manifolds $B$. This will be a warm-up for the
higher dimensional examples presented in section~\ref{s:calabi}.

According to formula~(\ref{scalarsymp}) for the scalar curvature,
we are looking for symplectic potentials $s:\breve{P}\subset\R\to\R$
such that $s'' >0$ and
\[
-\left(\frac{1}{s''(x)} \right)'' = 2k\,,
\]
where $k\in\R$ denotes the Gauss curvature. This implies that
\[
s''(x) = - \frac{1}{k x^2 - 2bx - c}\,,\ b,c\in\R
\]
where $x\in\breve{P}\subset\R$ is such that $s''(x)>0$.

\subsection{Cylinders} 

Suppose that $k=b=0$. Then
\[
s''(x) = \frac{1}{c} > 0 \Rightarrow c>0
\quad\text{and}\quad x\in\R\,.
\]
This means that $P = \R$ and $s: \breve{P}=\R \to \R$ can be 
written as
\[
s(x) = \frac{x^2}{2c}\,.
\]
Hence
\[
B = \breve{B} = \breve{P}\times {\mathbb T}^1 = \R \times S^1 =
\left\{ (x,y): x\in\R\,,\ 
y\in\R/2\pi\Z\right\}
\]
and the metric is given in matrix form by
\[
\begin{bmatrix}
1/c & 0 \\
0 & c
\end{bmatrix}\,,  
\]
i.e. we get a flat cylinder of radius $\sqrt{c}$.

\subsection{Cones}

Suppose that $k=0$ and $b\ne 0$. Then, modulo a translation
and possible sign change in the action variable $x$, we can
assume that $c=0$, $b>0$ and
\[
s''(x) = \frac{1}{2bx} > 0 \Rightarrow x > 0\,.
\]
This means that $P = \left[0, +\infty\right[$ and 
$s: \breve{P}= \left]0,+\infty \right[ \to \R$ can be written as
\[
s(x) = \frac{1}{b}\cdot \frac{1}{2} x\log x \,.
\]
If $b=1$ this is the canonical symplectic potential giving the
flat Euclidean metric on $\R^2$ (cf. Example~\ref{ex:symppotR2n}).
In general, as explained in~\cite{Abr3}, this is the symplectic 
potential of a cone metric of angle $\pi b$ on $\R^2$, given in matrix 
form by
\[
\begin{bmatrix}
\frac{1}{2bx} & 0 \\
0 & 2bx
\end{bmatrix}\,.  
\]
When $b=1/p$ with $p\in\N$, this corresponds to an orbifold flat metric
on $\R^2/\Z_p$ (see~\cite{Abr3}).

\subsection{Footballs}

Suppose that $k>0$. Then, modulo a translation in the action variable $x$, 
we can assume that $b=0$ and
\[
s''(x) = \frac{1}{c - kx^2} > 0 \Rightarrow
c>0\  \text{and} \ -\sqrt{c/k}<x<\sqrt{c/k}\,.
\]
This means that $P = \left[-\sqrt{c/k}, \sqrt{c/k}\right]$ and 
$s: \breve{P}= \left]-\sqrt{c/k}, \sqrt{c/k} \right[ \to \R$ can be written as
\[
s(x) = \frac{1}{\sqrt{ck}}\cdot\frac{1}{2}
\left[ (x+\sqrt{c/k}) \log (x+\sqrt{c/k})
+ (-x+\sqrt{c/k}) \log (-x+\sqrt{c/k}) \right]\,.
\]
If $c=1/k$ this is the canonical symplectic potential giving the
smooth round european football metric of total area $4\pi/k$ and constant Gauss 
curvature $k$ on $\CP^1 \equiv {\mathbb S}^2$ (cf. Example~\ref{ex:symppotPn}).
In general, this is the symplectic potential of a singular american football 
metric of angle $\pi \sqrt{ck}$ at the ``poles".

\subsection{Hyperbolic Metrics}

Suppose that $k<0$. Then, modulo a translation in the action variable $x$, 
we can assume that $b=0$ and
\[
s''(x) = \frac{1}{c - kx^2}\,.
\]

If $c>0$ then $s''(x)>0\,,\ \forall\,x\in\R$, which means that
$P = \R$ and $s: \breve{P}=\R \to \R$ can be written as
\[
s(x) = \sqrt{\frac{-1}{ck}} \arctan \left(\sqrt{\frac{-k}{c}}\,x\right)\,.
\]
This is the symplectic potential of a metric of constant Gauss curvature
$k<0$ on
\[
B = \breve{B} = \breve{P}\times {\mathbb T}^1 = \R \times S^1 =
\left\{ (x,y): x\in\R\,,\ 
y\in\R/2\pi\Z\right\}\,,
\]
i.e. an hyperboloid.

If $c<0$ then
\[
s''(x) > 0 \Rightarrow x\in \left] -\infty, - \sqrt{c/k} \right[
\cup \left] \sqrt{c/k}, +\infty \right[\,.
\]
Hence, up to a sign change in the action variable $x$, we may
assume that $P = \left[\sqrt{c/k},+\infty\right[$ and
$s:\breve{P} = \left]\sqrt{c/k},+\infty\right[ \to\R$ can be
written as
\[
s(x) = \frac{1}{\sqrt{ck}}\cdot\frac{1}{2}
\left[ (x-\sqrt{c/k}) \log (x-\sqrt{c/k}) 
- (x+\sqrt{c/k}) \log (x+\sqrt{c/k}) \right] \,.
\]
If $c=1/k$ this is the symplectic potential of the hyperbolic
metric of constant Gauss curvature $k<0$ on $\R^2$. In other words,
in the action-angle coordinates $(x,y)$ of this symplectic model,
the hyperbolic metric is given in matrix form by
\[
\begin{bmatrix}
\frac{-k}{(kx)^2 -1} & 0 \\
0 & \frac{(kx)^2 -1}{-k}
\end{bmatrix}\,.  
\]
More generally, i.e. when $c\ne 1/k$, we get singular hyperbolic 
metrics on $\R^2$, with a cone singularity of angle $\pi\sqrt{ck}$ at
the origin.

\begin{remark} \label{rem:complete}
This case illustrates the fact that there is no immediate relation 
between completeness of a toric compatible complex structure and 
completeness of the associated toric K\"ahler metric.
Here the metric is complete but the complex structure is not.
In fact, it easily follows from~(\ref{eq:holcoords}) that $\R^2$ with 
this complex structure is biholomorphic to an open bounded disc 
$D\subset\C$.
\end{remark}

If $c=0$ we have that
\[
s''(x) = \frac{1}{-kx^2} > 0\,,\ \forall\,x\ne 0\,,
\]
and $s:\breve{P} = \left]0,+\infty\right[ \to \R$ can be written as
\[
s(x) = \frac{1}{k} \log (x)\,.
\]
This is the symplectic potential of a complete hyperbolic cusp
metric on $\breve{B} = \left]0,+\infty\right[ \times S^1$,
given in matrix form by
\[
\begin{bmatrix}
\frac{1}{-k x^2} & 0 \\
0 & -k x^2
\end{bmatrix}\,.  
\]

\section{Calabi's Family of Extremal K\"ahler Metrics}
\label{s:calabi}

In~\cite{C2}, Calabi introduced the notion of extremal K\"ahler 
metrics. These are defined, for a fixed closed complex manifold 
$(M,J)$, as critical points of the square of the $L^2$-norm of 
the scalar curvature, considered as a functional on the space of all 
symplectic K\"ahler forms $\omega$ in a fixed K\"ahler class 
$\Omega\in H^2(M,\R)$. The 
extremal Euler-Lagrange equation is equivalent to the gradient of the 
scalar curvature being an holomorphic vector field (see~\cite{C1}), and 
so these metrics generalize constant scalar curvature K\"ahler metrics. 
Moreover, Calabi showed in~\cite{C3} that extremal K\"ahler metrics are 
always invariant under a maximal compact subgroup of the group of 
holomorphic transformations of $(M,J)$. Hence, on a toric manifold, 
extremal K\"ahler metrics are automatically toric K\"ahler metrics, and 
one should be able to write them down using the previous action-angle
coordinates framework. 

In this section, following~\cite{Abr1}, we will do that for the $4$-parameter 
family of $U(n)$-invariant extremal K\"ahler metrics constructed by Calabi 
in~\cite{C2}. Calabi used this family to put extremal K\"ahler metrics on
\[
\CP ({\mathcal O}(-m)\oplus\C)\longrightarrow \CP^{n-1}
\,,\ n,m \in\N\,,
\]
for any K\"ahler class. In particular, when $n=2$, on all Hirzebruch 
surfaces (cf. Example~\ref{ex:hirzebruch}). As we will see here, this family 
can be used to write down many other interesting extremal K\"ahler metrics, 
including the non-compact, cohomogeneity one, constant scalar curvature examples
that were later constructed by LeBrun~\cite{LeB}, Pedersen-Poon~\cite{PePo} 
and Simanca~\cite{Si}. Using the action-angle coordinates set-up for toric Sasaki 
geometry developed in~\cite{MSY}, one can show~\cite{Abr4} that Calabi's family 
also contains a family of K\"ahler-Einstein metrics directly related 
to the Sasaki-Einstein metrics constructed in 2004 by 
Gauntlett-Martelli-Sparks-Waldram~\cite{GW1,GW2}. 

\subsection{Calabi's Family in Action-Angle Coordinates}

Consider symplectic potentials $s:\breve{P}\subset(\R^+)^n \to \R$ of the 
form
\begin{equation} \label{eq:sympot-r}
s(x) = \frac{1}{2}\left(\sum_{i=1}^n x_i \log x_i  + h(r)\right)\,,
\end{equation}
where
\[
r=x_1 + \cdots x_n
\]
and $\breve{P}$ will be determined in each of the particular cases that
we will consider. A simple computation shows that
\[
\Det (S) = \frac{1+r h''(r)}{2^n x_1 \cdot\cdots\cdot x_n}
\quad\text{and}\quad 
S^{-1} = \left(s^{ij}= 
2 \left(\delta_{ij} x_i - x_i x_j f(r)\right)\right)\,,
\]
where $f = h''/(1 + r h'')$.
It then follows from~(\ref{scalarsymp}) that the scalar curvature of
the corresponding toric K\"ahler metric is given by
\begin{equation} \label{eq:scalar-r}
Sc (x) = Sc (r) = 2r^2f''(r) + 4(n+1)rf'(r) + 2n(n+1)f(r)\,.
\end{equation}

The Euler-Lagrange equation defining an extremal K\"ahler metric 
can be shown to be equivalent to
\begin{equation} \label{extremalsymp}
\frac{{\partial}Sc}{{\partial}x_j} \equiv\ \mbox{constant},\ j=1,\ldots,n,
\end{equation}
i.e. the metric is extremal if and only if its scalar curvature $Sc$ 
is an affine function of $x$ (see~\cite{Abr1}). 

Requiring that the scalar curvature $Sc \equiv Sc(r)$, given 
by~(\ref{eq:scalar-r}), is an affine function of $r$ is easily seen 
to be equivalent to
\begin{equation} \label{eq:h}
h''(r) = - \frac{1}{r} + \frac{r^{n-1}}{r^n - A - Br - Cr^{n+1} -
Dr^{n+2}}\,,
\end{equation}
where $A,B,C,D\in\R$ are the $4$ parameters of the family.

As shown by Calabi in~\cite{C2}, one can determine explicit values
for the constants $A,B,C,D\in\R$ so that the corresponding symplectic
potential, given by~(\ref{eq:sympot-r}), gives rise to an extremal
K\"ahler metric on
\[
H^n_m := \CP ({\mathcal O}(-m)\oplus\C)\longrightarrow \CP^{n-1}
\,,\ n,m \in\N\,.
\]

In our framework, this can be seen as follows. Up to a $GL(n,\R)$
transformation, generalizing to higher dimensions the one considered
in Example~\ref{ex:hirzebruch}, $H^n_m$ is determined by a moment
polytope $P^n_m (a,b) \subset\R^n$ with defining affine functions
\[
\ell_i (x) = x_i\,,\ \forall\,i=1,\ldots,n\,,
\quad
\ell_{n+1} (x) = \frac{1}{m} (r-a)
\quad\text{and}\quad
\ell_{n+2} (x) = \frac{1}{m} (b-r)\,,
\]
where the real numbers $0<a<b$ determine the K\"ahler class, i.e.
the cohomology class of the symplectic form $\omega_{a,b}$. Hence, 
if follows from Theorem~\ref{thm:Abr2} that, to determine a toric 
compatible complex structure on $(H^n_m, \omega_{a,b})$, the 
symplectic potential has to be of the form
\begin{equation} \label{eq:sympot-P}
2s(x) = \sum_{i=1}^n x_i \log x_i  + 
\frac{1}{m} ((r-a)\log(r-a) + (b-r)\log(b-r)) + 
\tilde{h}(r)\,,
\end{equation}
where $\tilde{h}$ is smooth on $P^n_m (a,b) \subset\R^n$. 
Comparing~(\ref{eq:sympot-r}), (\ref{eq:h}) and~(\ref{eq:sympot-P}),
one concludes that we must have
\begin{equation} \label{eq:system}
\frac{r^{n-1}}{r^n - A - Br - Cr^{n+1} - Dr^{n+2}} =
\frac{1}{m}\left(\frac{1}{r-a} + \frac{1}{b-r}\right) + R(r)\,,
\end{equation}
where $R(r)$ is a smooth function on $P^n_m (a,b) \subset\R^n$.
This gives rise to a system of $4$ linear equations in the $4$ unknowns
$A,B,C,D\in\R$, which admits a unique explicit solution for any
$n,m\in\N$ and $a,b\in\R$ such that $0<a<b$ (see page~285 
of~\cite{C2} or~\cite{Ra}).

\subsection{Particular Cases}

By construction, all K\"ahler metrics in Calabi's $4$-parameter family 
are extremal. A simple computation shows that their scalar curvature is given 
by
\[
Sc(r) = 2 (n+1) ((2+n)Dr+nC)\,.
\]
Hence, these metrics have 
\[
\text{constant scalar curvature iff $D=0$} 
\]
and are
\[
\text{scalar-flat iff $C=D=0$.} 
\]
Moreover, one can show that these metrics are 
\[
\text{K\"ahler-Einstein iff $B=D=0$}
\]
and 
\[
\text{Ricci-flat iff $B=C=D=0$.}
\]

We will now analyse in more detail these constant scalar curvature
particular cases. Note that when $A=B=C=D=0$ we have
\[
s(x) = \frac{1}{2} \sum_{i=1}^n x_i \log x_i\,,
\]
which is the standard symplectic potential of the Delzant set
$P = \left(\R^+_0 \right)^n$ and determines the standard flat
Euclidean metric on $\R^{2n}$ (cf. Example~\ref{ex:symppotR2n}).

\subsection{Ricci-Flat Metrics}

Assume that $B=C=D=0$ and $A=a^n$ with $0<a\in\R$. Then
\begin{align}
h''(r) & = -\frac{1}{r} + \frac{r^{n-1}}{r^n - a^n} \notag \\
& = -\frac{1}{r} + \frac{r^{n-1}}{(r-a) \sum_{k=1}^n a^{k-1} r^{n-k}} \notag \\
& = -\frac{1}{r} + \frac{1}{n}\cdot\frac{1}{r-a} + R(r)\,,\notag
\end{align}
where $R(r)$ is a smooth function on the rational Delzant
set $P^n (a) \subset\R^n$ with defining affine functions
\[
\ell_i (x) = x_i\,,\ \forall\,i=1,\ldots,n\,,
\quad\text{and}\quad
\ell_{n+1} (x) = \frac{1}{n} (r-a)\,.
\]
The symplectic potential can be written as
\[
s(x) = \frac{1}{2}\left(\sum_{i=1}^n x_i \log x_i  + 
\frac{1}{n} (r-a)\log(r-a)  + 
\tilde{h}(r)\right)\,,
\]
where $\tilde{h}$ is smooth on $P^n (a) \subset\R^n$. 
Hence, for each $a>0$, it defines a Ricci-flat K\"ahler metric
on the total space of the canonical line bundle
\[
{\mathcal O}(-n)\longrightarrow \CP^{n-1}
\]
(as before, up to a $GL(n,\R)$ transformation, the underlying
non-compact toric symplectic manifold is determined by 
$P^n(a)\subset\R^n$). These are the metrics constructed by
Calabi in~\cite{C1}.

\subsection{Scalar-Flat Metrics}
\label{ssec:scalar-flat}

We will now show that Calabi's family also contains the complete
scalar-flat K\"ahler metrics on the total space of the line bundles
\[
{\mathcal O}(-m)\longrightarrow \CP^{n-1}\,,\ \forall\,m\in\N\,,
\]
constructed by LeBrun~\cite{LeB} and Pedersen-Poon~\cite{PePo} 
(see also Simanca~\cite{Si}).

Up to a $GL(n,\R)$ transformation, these spaces are determined
by the rational Delzant sets $P^n_m (a) \subset \R^n$, 
with $0<a\in\R$ and defining affine functions
\[
\ell_i (x) = x_i\,,\ \forall\,i=1,\ldots,n\,,
\quad\text{and}\quad
\ell_{n+1} (x) = \frac{1}{m} (r-a)\,.
\]
which means that the symplectic potential has to be of the form
\[
s(x) = \frac{1}{2}\left(\sum_{i=1}^n x_i \log x_i  + 
\frac{1}{m} (r-a)\log(r-a)  + 
\tilde{h}(r)\right)\,,
\]
where $\tilde{h}$ is smooth on $P^n_m (a)$. This implies
that
\[
h''(r) = -\frac{1}{r} + \frac{1}{m}\cdot\frac{1}{r-a} + R(r)\,,
\]
with $R(r)$ smooth on $P^n_m (a)$. Since the scalar-flat condition
is equivalent to $C=D=0$, we get that
\[
- \frac{1}{r} + \frac{r^{n-1}}{r^n - A - Br} = 
-\frac{1}{r} + \frac{1}{m}\cdot\frac{1}{r-a} + R(r)\,.
\]
This relation gives rise to a system of $2$ linear equations in the 
$2$ unknowns $A,B\in\R$, which admits
a unique solution for any $n,m\in\N$ and $0<a \in\R$:
\[
A = a^n (1-n+m) \quad\text{and}\quad B = (n-m) a^{n-1}\,.
\]

Note that when $m=1$ we get complete scalar-flat K\"ahler metrics on
the total space of the line bundle
\[
{\mathcal O}(-1)\longrightarrow \CP^{n-1}\,,
\]
i.e. on $\C^n$ blown-up at the origin. These were originally
constructed by D.~Burns (at least when $n=2$).

\subsection{Fubini-Study and Bergman Metrics}

Assume that $A=B=D=0$, which implies in particular that we are
considering K\"ahler-Einstein metrics. Then
\begin{align}
h''(r) & = -\frac{1}{r} + \frac{r^{n-1}}{r^n - C r^{n+1}} \notag \\
& = -\frac{1}{r} + \frac{1}{r (1-Cr)} \notag \\
& = \frac{1}{\frac{1}{C} - r}\,,\notag
\end{align}
which implies that the symplectic potential can be written as
\[
s(x) = \frac{1}{2}\left(\sum_{i=1}^n x_i \log x_i  + 
\left|\frac{1}{C} - r \right|\log \left|\frac{1}{C} - r \right|\right)\,.
\]

When $C=1$ we recover Example~\ref{ex:symppotPn}, i.e. the standard
complex structure and Fubini-Study metric on $\CP^n$. More generally, 
for any $C>0$, this defines the standard complex structure and suitably 
scaled Fubini-Study metric on $\CP^n$. The corresponding moment polytope 
is the simplex in $\R^n$ with defining affine functions
\[
\ell_i (x) = x_i\,,\ \forall\,i=1,\ldots,n\,,
\quad\text{and}\quad
\ell_{n+1} (x) = \frac{1}{C} - r\,.
\]

When $C<0$ it follows from Theorem~\ref{thm:Abr2} that the above
symplectic potential determines a toric compatible complex structure
$J_C$ on the toric symplectic manifold $(\R^{2n}, \omega_{\rm st})$
with Delzant set $P=(\R^+_0)^n \subset \R^n$. The corresponding
K\"ahler metric is a $U(n)$-invariant K\"ahler-Einstein metric of
negative scalar curvature on the complex manifold $(\R^{2n}, J_C)$.
Using the holomorphic coordinates given by~(\ref{eq:holcoords}),
one easily concludes that $(\R^{2n}, J_C)$ is biholomorphic to a
ball $B\subset\C^n$, which implies in particular that $J_C$ is not
complete. Moreover, the K\"ahler metric
$\langle\cdot,\cdot\rangle_C : = \omega_{\rm st} (\cdot, J_C \cdot)$
is, in fact, the well-known and complete Bergman metric.

\subsection{Other K\"ahler-Einstein Metrics}

We will now show that Calabi's family also contains the
complete K\"ahler-Einstein metrics with negative scalar 
curvature on the total space of the open disc bundles 
\[
{\mathcal D} (-m) \subset {\mathcal O}(-m)\longrightarrow \CP^{n-1}\,,\ 
\forall\,n<m\in\N\,,
\]
constructed by Pedersen-Poon~\cite{PePo}.

As toric symplectic manifolds, and up to a $GL(n,\R)$ transformation, 
these spaces are again determined by the rational Delzant sets 
$P^n_m (a) \subset \R^n$ with defining affine functions
\[
\ell_i (x) = x_i\,,\ \forall\,i=1,\ldots,n\,,
\quad\text{and}\quad
\ell_{n+1} (x) = \frac{1}{m} (r-a)\,.
\]
which means that the symplectic potential has to be of the form
\[
s(x) = \frac{1}{2}\left(\sum_{i=1}^n x_i \log x_i  + 
\frac{1}{m} (r-a)\log(r-a)  + 
\tilde{h}(r)\right)\,,
\]
where $\tilde{h}$ is smooth on $P^n_m (a)$. This implies
that
\[
h''(r) = -\frac{1}{r} + \frac{1}{m}\cdot\frac{1}{r-a} + R(r)\,,
\]
with $R(r)$ smooth on $P^n_m (a)$. Since the K\"ahler-Einstein 
condition is equivalent to $B=D=0$, we get that
\[
- \frac{1}{r} + \frac{r^{n-1}}{r^n - A - Cr^{n+1}} = 
-\frac{1}{r} + \frac{1}{m}\cdot\frac{1}{r-a} + R(r)\,.
\]
This relation gives rise to a system of $2$ linear equations in the 
$2$ unknowns $A,C\in\R$, which admits
a unique solution for any $n<m\in\N$ and $0<a \in\R$:
\[
A = \frac{(m+1) a^n}{n+1} > 0
\quad\text{and}\quad 
C = \frac{n-m}{(n+1)a} < 0\,.
\]

As remarked by Pedersen-Poon, these metrics are a superposition of
Calabi's Ricci-flat metrics ($A>0$) and Bergman metrics ($C<0$). The 
analogous superposition of Calabi's Ricci-flat metrics ($A>0$) with
Fubini-Study metrics ($C>0$) gives rise to K\"ahler-Einstein metrics on 
the projectivization of the above line bundles, with cone-like singularities
in the normal directions to the zero and infinity sections. As explained
in~\cite{Abr4}, these metrics are directly related to the smooth 
Sasaki-Einstein metrics constructed in 2004 by 
Gauntlett-Martelli-Sparks-Waldram~\cite{GW1,GW2}.

\subsection{Other Constant Scalar Curvature Metrics}

We will now show that Calabi's family also contains the complete constant 
negative scalar curvature K\"ahler metrics on the total space of the open disc 
bundles 
\[
{\mathcal D} (-m) \subset {\mathcal O}(-m)\longrightarrow \CP^{n-1}\,,\ 
\forall\,n,m\in\N\,,
\]
constructed again by Pedersen-Poon~\cite{PePo}.

As before, we are interested in the rational Delzant sets 
$P^n_m (a) \subset \R^n$ and symplectic potentials of the form
\[
s(x) = \frac{1}{2}\left(\sum_{i=1}^n x_i \log x_i  + 
\frac{1}{m} (r-a)\log(r-a)  + 
\tilde{h}(r)\right)\,,
\]
with $\tilde{h}$ smooth on $P^n_m (a)$. Assuming $D=0$ and $C=-1$, 
i.e. $Sc = - 2n(n+1)$, this implies that
\[
- \frac{1}{r} + \frac{r^{n-1}}{r^{n} - A - Br + r^{n+1}} = 
-\frac{1}{r} + \frac{1}{m}\cdot\frac{1}{r-a} + R(r)\,,
\]
with $R(r)$ smooth on $P^n_m (a)$. This relation gives rise to a system of 
$2$ linear equations in the $2$ unknowns $A,B\in\R$, which admits
a unique solution for any $n,m\in\N$ and $0<a \in\R$:
\[
A = (m-n+(1-n)a)a^n
\quad\text{and}\quad 
B = (n-m+1+na)a^{n-1}\,.
\]

\end{document}